\documentclass[reqno]{siamart190516}
\usepackage{hyperref,cleveref}  
\usepackage{braket,amsfonts}
\usepackage{amssymb}
\usepackage{mathrsfs}
\usepackage{array}
\usepackage{stmaryrd}
\usepackage{titlesec}
\usepackage{enumerate}
\usepackage{url}
\usepackage{xr}
\titleformat{\paragraph}
{\normalfont\normalsize\bfseries}{\theparagraph}{1em}{}
\titlespacing*{\paragraph}
{0pt}{3.25ex plus 1ex minus .2ex}{1.5ex plus .2ex}
\usepackage{algorithm}
\usepackage{algpseudocode}
\usepackage{algorithmicx}

\usepackage{graphics}
\usepackage{graphicx}
\usepackage{caption,subcaption}
\usepackage{multicol}
\newtheorem{Assumption}{Assumption}[section]
\begin{document}
\title{Randomized Optimal Switching Problem and Related Mirror Descent Flow}
	\thanks{Y. Dong was supported by the National Natural Science Foundation of China (No.12471425)}
\author{Yuchao Dong
	\thanks{School of Mathematical Sciences, Tongji University \& Key Laboratory of Intelligent Computing and Applications (Tongji University), Ministry of Education,  Shanghai 200092, China (\email{ycdong@tongji.edu.cn})}}
\maketitle
\begin{abstract}
We study continuous-time reinforcement learning for the optimal switching problem, 
in which a decision-maker controls a diffusion process by switching among finitely 
many regimes, incurring both running and transition costs. To enable 
exploration, we relax the classical deterministic switching control to a randomized 
framework, where the switching decisions are governed by a continuous-time Markov 
chain with state-dependent generator, and augment the cost functional with a 
KL-divergence regularization weighted by a temperature parameter $\lambda$. Under 
mild assumptions on the coefficients, we establish that the regularized value 
function is the unique smooth solution of an elliptic Hamilton--Jacobi--Bellman 
system, and derive an explicit optimal Gibbs policy given by an exponential 
transformation of the value function differences across modes. We further prove 
that the regularized value function approximates the classical optimal value 
function with error of order $O\left(\lambda \log \frac{1}{\lambda}\right)$, 
which is consistent with analogous bounds established in other entropy-regularized 
control problems and is believed to be sharp. To solve the regularized problem 
numerically, we introduce a mirror descent flow in the dual logarithmic policy 
space, prove its well-posedness and the monotonic decrease of the value function 
along the flow, and establish quantitative error bound to the classical 
optimal value function. For a constant temperature scheduler, the convergence 
rate is of order $O\left(\frac{1}{e^{\lambda s} - 1}+\lambda \log\frac1\lambda\right)$, while under the 
annealing scheduler $\lambda_s = \frac{1}{\sqrt{1+s}}$, we obtain the rate 
$O\left(\frac{\log s}{\sqrt{s}}\right)$, which decays to zero as the flow 
time $s \to \infty$.
\end{abstract}

\section{Introduction}
Optimal switching problems arise naturally in a wide range of applications in finance, engineering, and operations research, where an agent must decide when and how to switch a process between finitely many regimes, balancing running costs against fixed switching costs. Classical examples include evaluating natural resource investments \cite{brennan1985evaluating}, the management of power plants \cite{carmona2010valuation}, investment decision with switching cost \cite{duckworth2001model}, and so on. The mathematical framework involves a controlled diffusion process whose drift/volatility coefficients depend on the current regime, with a finite set of modes and a cost associated with each transition. The main problem is to determine the optimal switching times and target modes so as to minimize the total expected discounted cost over a finite or infinite horizon.

The classical formulation of the optimal switching problem is well understood. Under appropriate assumptions, the value function is characterized as the unique viscosity solution of a system of Hamilton–Jacobi–Bellman (HJB) variational inequalities of obstacle type, see \cite{tang1993finite}.  This was established rigorously in many works such as Djehiche, Hamadène, and Popier \cite{djehiche2009finite}, Pham \cite{pham2007smooth} and Pham, Vath, and Zhou \cite{pham2009optimal}, using both PDE methods and backward stochastic differential equation (BSDE) techniques. The structure of the optimal stopping regions — where it is optimal to switch — and the free boundaries separating them have been analyzed in detail in these and subsequent works. 

Despite this solid theoretical foundation, computing the value function and optimal switching policy in practice remains a formidable challenge. Grid-based quasi-variational inequality solvers are effective only in low dimension. Regression-based simulation methods, pioneered by Longstaff and Schwartz \cite{longstaff2001valuing} for American options, have been adapted to the switching setting by A\"\i d et al \cite{aid2014probabilistic}. More recently, Bayraktar, Cohen, and Nellis \cite{bayraktar2023neural} developed a backward-in-time machine learning algorithm that uses a sequence of neural networks to solve optimal switching problems in energy production. These computational difficulties also motivated the development of reinforcement learning methods that  learn policies directly from sampled trajectories of the controlled system.

\subsection{Literature Review}
The application of reinforcement learning to continuous time stochastic control has been an active area of research over the past several years. A key conceptual step, introduced by Wang, Zariphopoulou, and Zhou \cite{wang2020reinforcement}, is to relax the classical deterministic control to a randomized or exploratory control, adding an entropy regularization term to the objective to encourage exploration. This relaxation turns the original optimization problem into a smoother one, and the resulting regularized value function satisfies an HJB equation with a closed-form optimal policy given by a Gibbs (softmax) distribution. This exploratory framework has subsequently been extended and analyzed in a variety of settings. Jia and Zhou developed a systematic policy gradient theory for continuous-time RL, providing convergence guarantees for policy gradient algorithms based on martingale formulations of the policy gradient in a series of paper \cite{jia2022policy1,jia2022policy2,jia2023q}. Tang, Zhang, and Zhou \cite{tang2022exploratory} studied the solution of exploratory Hamilton–Jacobi–Bellman equations and proved its convergence to the classical value function when the temperature constant $\lambda$ goes to zero. 

 Most of previous research concentrated on absolutely-continuous control. In the context of optimal control theory, there is a large and important class of problems called singular control, where the  effect of the control on the state is immediate including  irreversible investment, dividend optimization, and optimal stopping. Dong \cite{dong2024randomized} was the first to study the optimal stopping problem in continuous time within an exploratory framework, formulating the randomized stopping time as a jump intensity control problem and obtaining a  $O(\lambda \log\frac1\lambda)$ error bound  between the regularized and classical value functions.  Dai et al \cite{dai2026learning}  transformed the stopping problem into a two-action stochastic control problem, randomize controls into Bernoulli distributions with entropy regularization. Dianetti, Ferrari, and Xu \cite{dianetti2024exploratory} formulate the exploratory optimal stopping problem by representing the randomized stopping time as a bounded non-decreasing càdlàg process and penalizing the performance criterion with the cumulative residual entropy of the stopping time; the regularized problem takes the form of a degenerate singular stochastic control.  
 
 The optimal switching problem itself has very recently entered the continuous-time RL literature, with two  works appearing in close proximity to the present paper. Dai, Dong, and Li \cite{dai2025reinforcement} propose a reinforcement learning approach to find arbitrage strategies in stock index futures, formulated as an optimal switching problem among three regimes (long, flat, and short). Huang et al \cite{huang2025continuous} studied continuous-time RL for optimal switching problems across multiple regimes. Despite their contributions, several important questions remain open. The entropy-regularized model studied in \cite{dai2026learning} is limited to a special three-regime setting, making it difficult to assess whether the underlying ideas persist in more general switching systems. On the other hand, the randomized switching framework of \cite{huang2025continuous} accommodates general multi-regime problems, but its analysis is largely qualitative. In particular, the asymptotic effect of entropy regularization is not quantified, and no explicit error bound is available for the approximation of the classical optimal switching problem by its entropy-regularized counterpart.
 
 The mirror descent flow and related policy optimization methods in continuous-time stochastic control have recently received significant attention. In the discrete-time setting, mirror descent and natural policy gradient methods have been extensively studied, see e.g. Agarwal et al. \cite{agarwal2020optimality} and Lan \cite{lan2023policy}, where geometric convergence rates were established under various regularity conditions. The continuous-time analogue was developed systematically by Kerimkulov et al \cite{kerimkulov2025fisher}, who studied the global convergence of a Fisher-Rao policy gradient flow for infinite horizon entropy-regularised Markov decision processes with Polish state and action space. A particularly relevant contribution is due to Sethi, Siska, and Zhang \cite{sethi2025entropy},  who studied entropy annealing for policy mirror descent in continuous time and space. They proved that the dynamics converges  to the optimal solution of the regularized problem for different temperature schedulers. The present paper draws heavily on the structural insights of \cite{kerimkulov2025fisher} and \cite{sethi2025entropy}, adapting them to the switching control setting, where the policy space consists of generator matrices of continuous-time Markov chains rather than probability density controls, and where the interaction between different modes introduces additional coupling through the switching cost structure.

\subsection{Our Contribution} 
In our paper, we adopt the same exploratory framework to \cite{dai2025reinforcement} and \cite{huang2025continuous}. The decision-maker randomizes both switching time and the selection of the targeted regime state by invoking a generator matrix of an associated continuous-time Markov chain (CTMC) defined on finite state space. The entropy regularization on off-diagonal terms of the generator is imposed to encourage the exploration. Our major contributions are threefold.

First, we introduce a modified entropy regularization for randomized switching policies. Compared with previous works, the regularization term is modified from $\pi\log \pi-\pi$ to $\pi \log \pi -\pi+1$. Although this modification does not alter the optimal policy, it yields a clean path-space variational formulation in which the regularized objective is represented as a Kullback--Leibler divergence between the controlled switching law and a reference law. This interpretation reveals a close connection between  entropy-regularized control, and other KL-penalized optimization problems.

Second, we establish an explicit error estimate for the vanishing-regularization limit. We prove that the value function of the entropy-regularized switching problem converges to the value function of the classical optimal switching problem with order $O(\lambda \log\frac 1\lambda)$. To the best of our knowledge, this is the first explicit regularization-bias estimate for entropy-regularized stochastic switching problems. The same logarithmic correction has recently appeared in several KL-regularized optimization problems, suggesting that it reflects a generic feature of entropy regularization rather than a peculiarity of switching systems.

Third, we develop a mirror-descent framework for policy optimization in randomized switching problems. By parameterizing switching intensities through their logarithms, we derive a mirror-descent flow in policy space and prove its global well-posedness and monotonic improvement property. Combined with an annealing schedule for the regularization parameter, the proposed flow converges to the optimal value function of the original switching problem and admits explicit convergence estimates.

Together, these results establish a unified framework for entropy-regularized stochastic switching, encompassing its KL variational structure, quantitative approximation theory, and mirror-descent optimization dynamics.

The remainder of the paper is organized as follows. Section 2 sets up the classical and randomized switching problems, states the main assumptions, and establishes the well-posedness of the regularized value function. Section 3 quantifies the approximation error between the regularized and classical value functions. Section 4 introduces the mirror descent flow in the logarithmic policy space and proves its convergence. Section 5 contains the proofs of all main results.

{\bf Notations} Throughout this paper, let $\mathcal O\subset\mathbb{R}^d$ be a bounded domain with smooth boundary. For a smooth function on $\Omega$, $D_\alpha \varphi= \frac{\partial \varphi}{\partial x_\alpha}$ denotes the first-order partial derivative of $\varphi$ with respect to the coordinate $x_\alpha$, and $D_{\alpha\beta} \varphi  = \frac{\partial^2 \varphi}{\partial x_\alpha \partial x_\beta}$ denotes the second-order mixed partial derivative of $\varphi$. Thus, $D\varphi = \nabla\varphi$ and $D^2\varphi = \nabla^2\varphi$ stand for the gradient vector and the Hessian matrix of $\varphi$. For any vector $v$, $D_v \varphi$ denotes the directional derivative of $\varphi$ in the direction $v$. $C(\bar {\mathcal O})$ denotes the space of continuous functions on the closure of $\mathcal O$. $W^{2,p}(\mathcal O)$ stands for the Sobolev space on $\mathcal O$. In this paper, $C$ denotes a generic constant, whose value may differ from line to line.
\section{Problem Setup}
This section formulates the optimal switching problem in both its classical and 
randomized forms. We first recall the classical 
framework, where the decision-maker selects a sequence of stopping times and 
target modes to minimize a total expected discounted cost, and state the associated 
HJB variational inequalities characterizing the optimal value function. We then 
introduce  the randomized counterpart, where the 
switching control is replaced by a continuous-time Markov chain with state-dependent 
generator, augmented by a KL-divergence regularization term to encourage exploration.
\subsection{Classical Optimal Switching Problem}
Let $(\Omega,\mathcal F,P)$ be a complete probability space supporting a $d$-dimensional Brownian motion $W=(W_t)_{t\ge 0}$. Denote by $\mathbb F$ the complete and right-continuous filtration generated by $W$. The set $\mathcal A$ of admissible switching controls is the totality of double sequences $\xi=(\tau_k,\kappa_k)_{k\ge 0}$, where $(\tau_k)_{k\ge0}$ is a non-decreasing sequence of $\mathbb F$-stopping times such that $\tau_0=0$ and $\lim_{k} \tau_k=+\infty$; $\kappa_k$ is an $\mathcal F_{\tau_k}$-measurable random variable with values in $\mathbb I_m:=\{1,2,...,m\}$. For any $\xi\in \mathcal A$ and initial region $i \in \mathbb I_m$, we associate a process $(I_t)_{t \ge 0}$ as 
\begin{equation}\label{switching_control}
I_t=\sum_{k} \kappa_k 1_{t \in[\tau_k,\tau_{k+1})}, I_{0-}=i.
\end{equation}

Let $\mathcal O$ be an bounded open set in $\mathbb R^d$ with smooth boundary. Given $(x,i) \in \mathcal O \times \mathbb I_m$ and an admissible control $\xi \in \mathcal A$, the controlled dynamic $X$ is governed by the following SDE:
\begin{equation}
dX_t=b_{I_t}(X_t)dt+\sigma(X_t) dW_t,X_0=x,
\end{equation}  
where $(I_t)_{t\ge0}$ is the associate process defined by \eqref{switching_control}. Let $\tau_O$ be the fist exit time of $X$ from $\mathcal O$, which is defined as 
$$
\tau_{\mathcal O}=\inf\{t>0| X_t \notin \mathcal O\}.
$$
 The  total cost is given as 
 \begin{equation}
 J=\int_0^{\tau_\mathcal O} e^{-rt}L_{I_t}(X_t)dt+\sum_{k} e^{-r\tau_k}G_{\kappa_{k-1}\kappa_{k}}1_{\tau_k\le \tau_\mathcal O}.
 \end{equation}

 In this paper, we shall have the following assumptions on the coefficients.
 \begin{Assumption}
 \begin{enumerate}
 	\item The drift $b_i: \mathbb R^d  \rightarrow \mathbb R^d$, the diffusion tern $\sigma: \mathbb R^d \rightarrow \mathbb R^{d\times d}$ and the running cost $L_i: \mathbb R^d \rightarrow \mathbb R^+$are assumed to be  continuous differentiable and  the functions and their derivatives are bounded uniformly in $i$, i.e. there exists a constant $C_{\text{coef}}$ such that, for $\varphi=b_i,\sigma,L_i,Db_i,D\sigma$, and  $DL_i$ 
 $$
 			|\varphi(x)| \le C_{\text{coef}},
 $$	
  for any $x\in \mathcal O$ and $i \in \mathbb I_m$. 
 \item The volatility term is   non-degenerate, i.e. there exists a constant $\sigma_0>0$ such that, for any $x\in \mathcal O,$ and $\xi \in \mathbb R^d$,
 $$
 \xi \sigma\sigma^{T}(x)\xi^T \ge \sigma_0\xi\xi^T. 
 $$ 
 \item For any $i,j,k \in \mathbb I_m$ with $i\neq j$,  $G_{ij}>0$ and $G_{ij}+G_{jk}>G_{ik}$ with the convention that $G_{ii}=0$. 
 \item The discount rate $r>0$.
 \end{enumerate}
 \end{Assumption}
 The optimal value function is defined as 
 $$
 V^*_i(x)=\inf_{\xi \in \mathcal A}\mathbb E^{\xi}_{i,x}[J],
 $$
  where the notation $\mathbb E_{i,x}^\xi[\cdot]$ indicates that the initial state of the underlining process is $(i,x)$ and the switching control is $\xi$.
 Then, $(V^*_i)_{i\in \mathbb I_m}$ is the viscosity solution of the following HJB variational inequalities
 \begin{equation}\label{VI_sys}
 \begin{split}
\max\left\{ -\mathcal L_iV^*_i(x) -L_i(x)+rV^*_i(x),V^*_i(x)-\inf_{j\neq i } (V^*_j(x)+G_{ij})\right\}=0,x\in \mathcal O, i\in \mathbb I_m,
\end{split}
 \end{equation}
 with the boundary condition $V^*_i(x)=0$ for $x\in \partial \mathcal O$, and $ i\in \mathbb I_m$, where the operator $\mathcal L_i$ is defined as 
 $$
 \mathcal L_i \varphi(x)=\frac12 \sum_{\alpha,\beta=1}^da^{\alpha\beta}(x)D_{\alpha\beta}\varphi(x)+b^\alpha_i(x)D_{\alpha}\varphi(x),
 $$
 with $b_i^\alpha$ being the $\alpha$-th coordinate of $b_i$ and the matrix $A=(a^{\alpha\beta})(x):=\sigma\sigma^T(x)$. Moreover, in the region $\mathcal C_i:=\{ V_i(x) <\inf_{j\neq i } V_j(x)+G_{ij} \}$, $V^*_i$ is second order continuous differentiable.
\subsection{Randomized Switching Problem}
 The set $\mathcal A^r$ of admissible feedback randomized policies is the totality of all   $\boldsymbol{\pi}=(\pi^{ij})$ such that: (1) $\pi^{ij}(x)$ is a non-negative and bounded function on $\mathbb R^d$  for $i\neq j$; (2) $\pi^{ii}(x)=-\sum_{i\neq j} \pi^{ij}(x)$. 
 
 Given  $(i,x)\in \mathbb I_m \times \mathcal O$ and $\boldsymbol{\pi} \in \mathcal A^r$, the related pair of controlled processes $(X,I)$  are given as 
 \begin{equation}
 	dX_t=b_{I_t}(X_t)dt+\sigma(X_t) dW_t, X_0=x, 
 \end{equation}
 and $(I_t)_{t\ge 0}$ is a continuous-time finite state Markov chain with state space $\mathbb I_m$ and initial state $i$ whose generator is $(\pi^{ij}(X_t))$.   For $i\neq j$, $\pi^{ij}(X_t) \ge 0$ is the instantaneous intensity of the transition of  $I$ from state $i$ to state $j$ at time $t$. With a slight  abuse of notation, denote by $\tau_k$ the $k$-th jump time of $I$ and $\kappa_k:=I_{\tau_k}$.    The corresponding regularized expected cost is defined as 
 $$
 	J^\lambda(\boldsymbol{\pi})= \int_0^{\tau_{\mathcal O}} e^{-rt}\left(L_{I_t}(X_t) +\lambda \sum_{j\neq I_t} D(\pi^{I_tj}(X_t)|1)\right)dt+\sum_{k}e^{-r\tau_k} G_{\kappa_{k-1}\kappa_{k}}1_{\tau_k \le \tau_\mathcal O}, 
 $$
 Here, we add an extra term $D(x|y)=x\log\frac xy+y-x$, which serves as the KL-divergence in the work of  \cite{sethi2025entropy} \footnote{In fact, $D(x|y)$ is the KL-divergence between two exponential distributions whose intensities are $x$ and $y$  respectively.}.  The weight parameter $\lambda$ is often referred as temperature constant in the literature. 
 
Before proceeding with further theoretical discussion, let us further justify this formulation by Girsanov transformation. Consider one particular element $\boldsymbol{\pi}^0 $ in $\mathcal A^r$ such that $\pi^{ij} \equiv 1$ for $i \neq j$ and denote by $P^0$ the probability measure on the path space under this policy. For each $i \neq j$, let $N^{ij}_t$ count the transitions from $i$ 
to $j$ up to time $t$.  For any $\boldsymbol{\pi} \in \mathcal A^r$, consider the Radon--Nikodym density process $\mathscr E$ defined as 
\[
\mathscr E_t = \exp\!\left( \sum_{i \neq j} \int_0^t \log{\pi}^{ij}(X_s)
\mathbf{1}_{\{I_{s-} = i\}} \, dN^{ij}_s 
- \sum_{i \neq j} \int_0^t \bigl({\pi}^{ij}(X_s) -1\bigr) 
\mathbf{1}_{\{I_{s-} = i\}} \, ds \right).
\] 
Under our assumption on $\mathcal A^r$, $\mathscr E_t$ is a martingale and defines a new probability measure $P^{\boldsymbol{\pi}}$ as 
$$
\frac{dP^{\boldsymbol{\pi}}}{dP^0}\bigg|_{\mathcal F_t}=\mathscr E_t.
$$
From Girsanov theorem for point processes (see \cite{jacod2013limit}), one see that, under $P^{\boldsymbol{\pi}}$, $I_t$ is a continuous time Markov chain whose generator is $\boldsymbol{\pi}$. Moreover, for each $i \neq j$,
\[
{M}^{ij}_t := N^{ij}_t - \int_0^t {\pi}^{ij}(X_s) 
\mathbf{1}_{\{I_{s-}=i\}} \, ds
\]
is a  martingale. Let $\tau$ be an independent exponential distributed random variable, whose intensity is $r$. Then, one can calculate the KL divergence between  $P^{\boldsymbol{\pi}}$ and $P^0$ restricted to $\mathcal F_{\tau_{\mathcal O} \wedge \tau}$ as 
\begin{equation*}
	\begin{split}
&KL(P^{\boldsymbol{\pi}}|P^0)=\mathbb E^{\boldsymbol{\pi}}\left[\mathscr E_{\tau_{\mathcal O}\wedge \tau}\right]  \\
=&\mathbb E^{\boldsymbol{\pi}}\left[ \sum_{i \neq j} \int_0^{\tau_{\mathcal O} \wedge \tau} \log{\pi}^{ij}(X_s) 
\mathbf{1}_{\{I_{s-} = i\}} \, dN^{ij}_s 
- \sum_{i \neq j} \int_0^{\tau_{\mathcal O} \wedge \tau} \bigl({\pi}^{ij}(X_s) -1\bigr) 
\mathbf{1}_{\{I_{s-} = i\}} \, ds\right] \\
=&\mathbb E^{\boldsymbol{\pi}}\left[ \int_0^{\tau_{\mathcal O} \wedge \tau}\sum_{j \neq I_s} (\pi^{I_sj}_s\log{\pi}^{I_sj}(X_s)-\pi^{I_sj}(X_s)+1) ds\right]\\
=&\mathbb E^{\boldsymbol{\pi}}\left[ \int_0^{\tau_{\mathcal O} \wedge \tau}\sum_{j \neq I_s} D(\pi^{I_sj}(X_s)|1) ds\right]=\mathbb E^{\boldsymbol{\pi}}\left[ \int_0^{\tau_{\mathcal O}}e^{-rs}\sum_{j \neq I_s} D(\pi^{I_sj}(X_s)|1) ds\right].
\end{split}
\end{equation*}
This gives the key probabilistic interpretation: the additional regularization term in the cost functional is exactly the KL-divergence of the controlled path measure with respect to the reference measure with unit intensities. The regularized cost functional can be written as
$$
 \mathbb E_{i,x}^{\boldsymbol{\pi}}[J^\lambda(\boldsymbol{\pi})]= \mathbb E_{i,x}^{\boldsymbol{\pi}}[J]+\lambda KL(P^{\boldsymbol{\pi}}|P^0),
$$
 where $\mathbb E_{i,x}^{\boldsymbol{\pi}}[\cdot]$ indicates that the initial state is $(i,x)$ and the governing policy is $\boldsymbol{\pi}$.

 The optimal value function is defined as 
 $$
 V^\lambda_i(x)=\inf_{\boldsymbol{\pi} \in \mathcal A^r} \mathbb E_{i,x}^{\boldsymbol{\pi}}[J^\lambda(\boldsymbol{\pi})].
 $$
By the  dynamic programming principle,  the value function $(V^\lambda_i)_{i\in \mathbb I_m}$ satisfies the following HJB system
 \begin{equation}\label{hjb_system}
 \mathcal L_i V^\lambda_i(x)+L_i(x)-rV^\lambda_i(x)=\lambda \sum_{i \neq j} (\exp(\frac{V^\lambda_i(x)-V^\lambda_j(x)-G_{ij}}{\lambda}) -1), x\in \mathcal O,
 \end{equation}
 with boundary condition $V^\lambda_i(x)=0$ on $\partial \mathcal O$. Then, we have the following theorem
 \begin{theorem}\label{thm_hjb_system}
There exists a unique  solution $(V^\lambda_i)_{i \in \mathbb I_m}$ with $V^\lambda_i \in C^2(\mathcal O) \cap C(\bar {\mathcal O})$ for HJB system \eqref{hjb_system}. Moreover,  it holds that 
\begin{equation}\label{ineq_l_infty_esti}
0 \le V^\lambda_i(x) \le \frac{C_{\text{coef}}+(m-1)\lambda}{r}.	
\end{equation} 	
The optimal feedback  policy $\bar {\boldsymbol{\pi}}=(\bar \pi^{ij}(x))$ given by 
\begin{equation}\label{optim_cond}
\bar \pi^{ij}(x) =\exp(\frac{V^\lambda_i(x)-V^\lambda_j(x)-G_{ij}}{\lambda}).	
\end{equation}
 
 \end{theorem}
 \section{The Difference between $V^\lambda$ and  $V^*$}
 The entropy-regularized switching problem provides a smooth approximation of the classical optimal switching problem. The purpose of this section is to quantify the resulting regularization bias. More precisely, we derive an explicit estimate for the difference between the regularized value function $V^\lambda$ and the classical value function $V^*$ as $\lambda \to 0$. The main result shows that the approximation error is of order $O(\lambda \log(1/\lambda))$. To establish it, we first prove several auxiliary lemmas.
 \begin{lemma}\label{lem_lip}
 	There exists a constant $C_{Lip}$ depending on the coefficients but  independent of $\lambda$, such that, for any $x \in {\mathcal O}$ and $i \in I_m$, 
 	\begin{equation}
 		|D V^\lambda_i(x)| \le C_{Lip}.
 		\end{equation}
 \end{lemma}
 Then, we have the following lemma. 
 \begin{lemma}\label{lem_near_ineq}
 	For any $x \in {\mathcal O}$ and $i ,j\in I_m$,
 	\begin{equation}
 		V^\lambda_i(x) \le V^\lambda_j(x) +G_{ij}+\lambda \log\frac{C_{up}}{\lambda},
 		\end{equation}
  where $C_{up}=2C_{coef}(1+C_{Lip})$.
 \end{lemma}
 Note that a similar inequality has been given in \cite{dai2025reinforcement}. From the viewpoint of variational inequalities, the estimate quantifies the violation of the obstacle condition by the entropy-regularized value function. From the construction of the optimal policy $\bar {\boldsymbol{\pi}}$ in \eqref{optim_cond}, the previous lemma implies that it is upper bounded by $O(\frac{1}{\lambda})$, and the overall regularization term times the temperature constant $\lambda$ will be the order of $O(\lambda \log \frac1\lambda)$. Hence, it will implies the following result.  
 \begin{theorem}\label{thm_diff_V}
 	It holds that 
 	$$
 	\kappa(\lambda) V^\lambda_i-\frac{(1-\kappa(\lambda))C_{coef}+(m-1)\lambda}{r} \le V^*_i \le V^\lambda_i,
 	$$
 	where  $\kappa(\lambda)=\min_{ij} \frac{G_{ij}}{G_{ij}+\lambda\log\frac {C_{up}}{\lambda }}$. 
 \end{theorem}
 The most left hand side of above inequality can be rewritten as $$V^\lambda_i-(1-\kappa(\lambda))V^\lambda_i-\frac{(1-\kappa(\lambda))C_{coef}+(m-1)\lambda}{r}.$$ From the definition of $\kappa(\lambda)$, one sees that $$(1-\kappa(\lambda))=\sup_{ij} \frac{\lambda\log\frac {C_{up}}{\lambda }}{G_{ij}+\lambda\log\frac {C_{up}}{\lambda }}=O(\lambda\log \frac1\lambda).$$ Combining with the bound provided in Theorem \ref{thm_hjb_system}, Theorem \ref{thm_diff_V}  equivalently states that, for sufficiently small $\lambda$ and some constant $C$, 
 \begin{equation}\label{ineq_error}
 	 V^\lambda_i(x)-C\lambda \log \frac1\lambda\le V^*_i \le V^\lambda_i.
 \end{equation}

 Similar error bounds have been given in various optimal control problems, see \cite{dong2024randomized} for optimal stopping problem and \cite{sethi2025entropy} for drift control problem. We therefore believe that the bias order  $O(\lambda \log \frac1\lambda)$ is ubiquitous in a certain sense when an entropy or KL-divergence term is employed as a regularization. Strikingly, the identical rate arises independently in the theory of entropic optimal transport, a seemingly unrelated field. For example, Eckstein and Nutz \cite{eckstein2024convergence} established sharp convergence rates of this precise order for divergence-regularized optimal transport as the regularization parameter vanishes. 
 
 The similarity is not merely quantitative. Both problems can be viewed as optimization over probability laws with a relative-entropy penalty. In entropic optimal transport, one minimizes
 \[
 \int c(x,y)\, d\gamma(x,y)
 +\lambda\,\mathrm{KL}\!\left(\gamma\,\middle|\,\mu\otimes\nu\right)
 \]
 over admissible couplings $\gamma$ satisfying prescribed marginal constraints. In our setting, the control policy $\pi$ induces a path measure $\mathbb P^\pi$, and Section~2 shows that the regularized switching problem is equivalent to minimizing
 \[
 \mathbb E^\pi[J]
 +\lambda\,\mathrm{KL}\!\left(\mathbb P^\pi\,\middle|\,\mathbb P^0\right)
 \]
 over all path measures generated by admissible switching intensities. Thus, both frameworks consist of a linear cost functional acting on a probability measure together with a KL-divergence relative to a reference measure. The essential difference lies in the admissible set: entropic optimal transport constrains the marginals of a static coupling, whereas our problem constrains the measure to arise from a controlled drift-switching dynamics.
 
 This common variational structure also explains the appearance of Gibbs-type optimizers in both theories. In entropic optimal transport, the optimal coupling admits the Gibbs form
 \[
 d\gamma^\varepsilon
 =
 \exp\!\left(
 -\frac{c+\varphi+\psi}{\varepsilon}
 \right)
 \,d(\mu\otimes\nu),
 \]
 while the optimal switching intensity is given by
 \[
 \bar{\pi}_{ij}(x)
 =
 \exp\!\left(
 \frac{V_i^\lambda(x)-V_j^\lambda(x)-G_{ij}}{\lambda}
 \right).
 \]
 In both cases, the KL penalty converts a hard optimization problem into a smooth convex dual problem whose optimizer has an exponential (Boltzmann--Gibbs) structure. From this perspective, entropy-regularized stochastic control and entropic optimal transport may be viewed as dynamic and static manifestations of the same KL-penalized variational principle.
 
 \section{Mirror Descent Flow}
 Having established the approximation properties of the entropy-regularized switching problem in the previous section, we now turn to the algorithmic question of how to compute the associated optimal policy. Our goal is to develop a continuous-time policy optimization dynamics that exploits the KL-regularized structure of the problem and enjoys provable convergence guarantees. To this end, we first derive a performance difference identity that characterizes the variation of the value function under policy perturbations. This identity naturally suggests a gradient-type evolution in the space of switching intensities. Motivated by \cite{sethi2025entropy}, rather than working directly in the primal policy space, we introduce a logarithmic parameterization of the switching intensities and formulate a mirror descent flow in the corresponding dual space. This representation is particularly well suited to entropy regularization, preserves the positivity of switching rates, and leads to a monotone improvement property for the value function along the flow. In this section, we will establish the well-posedness of the mirror descent dynamics and derive quantitative convergence estimates under both constant and annealing temperature schedules.
 
  For any $\boldsymbol{\pi} \in \mathcal A^r$, one can compute that the corresponding value function $V^{\boldsymbol{\pi},\lambda}_i:=\mathbb E^{\boldsymbol{\pi}}_{i,x}[J^\lambda(\boldsymbol{\pi})]$ is the viscosity solution of  the following elliptic system
 \begin{equation}\label{eq_value_function}
 	\mathcal L_i V^{\boldsymbol{\pi},\lambda}_i+L_i(x)-r V^{\boldsymbol{\pi},\lambda}_i+\sum_{ j \neq i}(V^{\boldsymbol{\pi},\lambda}_j+G_{ij}-V^{\boldsymbol{\pi},\lambda}_i)\pi^{ij}+\lambda D(\pi^{ij}|1)=0,  
 \end{equation}
 with the boundary condition $V^{\boldsymbol{\pi},\lambda}_i(x)=0$ on $\partial \mathcal O$. Moreover, standard regularity results of linear PDEs yield that $V^{\boldsymbol{\pi},\lambda}_i\in W^{2,p^*}(\mathcal O)$ and $tr(\sigma\sigma^TD^2V^{\boldsymbol{\pi},\lambda}_i) \in L^\infty(\mathcal O)$.    
 
 Given any two policy $\boldsymbol{\pi}$ and $\tilde{\boldsymbol{\pi}}$, denote by $(X,I)$ the corresponding process under the policy $\boldsymbol{\pi}$. We first prove the following so-called performance difference lemma
 \begin{lemma}[Performance Difference]\label{lem_performance_diff}
 	\begin{small}
\begin{equation*}
	\begin{split}
&V^{\boldsymbol{\pi},\lambda}_i(x)-V^{\tilde{\boldsymbol{\pi}},\lambda}_i(x)=\lambda \mathbb E_{i,x}^{\boldsymbol{\pi}}\left[ \int_0^{\tau_{\mathcal O}}\sum_{j \neq I_t}(D(\pi^{I_t j}(X_t)|\tilde {\pi}^{I_t j}(X_t)) )e^{-rt}dt\right]\\
&+\mathbb E^{ {\boldsymbol{\pi}}}_{i,x}\left[\int_0^{\tau_{\mathcal O}}  \sum_{j \neq  I_{t}}(V^{\tilde{\boldsymbol{\pi}},\lambda}_j+G_{I_tj}-V^{\tilde{\boldsymbol{\pi}},\lambda}_i+\lambda \log\tilde \pi^{I_tj})(X_t)(\pi^{I_tj}(X_t)-\tilde{\pi}^{I_tj}(X_t))e^{-rt}dt  \right].
\end{split}
\end{equation*}
\end{small}
\end{lemma}

Motivated by the this equation, one may define a flow in $\{\boldsymbol{\pi}_s\}_{s\ge0} \subset \mathcal A^r$ as
\begin{equation*}
\partial_s \pi^{ij}_s(x)=-\left( V^{\boldsymbol{\pi},\lambda}_j(x)+G_{ij}-V^{\boldsymbol{\pi},\lambda}_i(x) +\lambda \log \pi^{ij}_s(x) \right)\pi^{ij}_s(x).
\end{equation*} 
This flow is known as the Fisher--Rao gradient flow, studied systematically by Kerimkulov et al \cite{kerimkulov2025fisher} in the context of entropy-regularized Markov 
decision processes on Polish spaces. The Fisher--Rao flow is a gradient flow on the space of policies equipped with the the natural Riemannian metric on the statistical manifold of probability distribution, which is closely related to the natural policy gradient method of Kakade \cite{kakade2001natural}. However, it turns out that working directly with the Fisher--Rao flow in the primal policy space presents several analytical difficulties in our setting. This motivates passing to the dual logarithmic policy space, which is  called the mirror descent flow. The terminology reflects the classical mirror descent algorithm of Nemirovski and Yudin \cite{nemirovskij1983problem}, in which the Bregman divergence associated to a strictly convex mirror map replaces the Euclidean distance in the gradient descent update. In our case, it means that we shall consider  the logarithm of $\pi^{ij}$.  Thus, we focus on a refined subset $\mathcal A^r_G$ of $\mathcal A^r$, which is defined as 
\begin{equation*}
\mathcal A^r_G=\{\boldsymbol{\pi}=(\pi^{ij}) | \pi^{ij}(x)=\exp(Z^{ij}(x) )\text{ for some $Z^{ij} \in C_b(\mathcal O)$, and $i \neq j$} \}.
\end{equation*}
We call the elements in $\mathcal A^r_G$ Gibbs policies. Given $Z=(Z^{ij})_{i\neq j} \in (C_b(\mathcal O))^{m\times (m-1)} $, it is obvious that one can define an element in $\mathcal A^r_G$ by exponential transformation, and we denote it as $\boldsymbol{\pi}(Z)$.

Now, consider a temperature schedule $\boldsymbol{\lambda} \in C^1([0,\infty),(0,\infty))$ with $\boldsymbol{\lambda}_s$ being non-increasing.  the related mirror-descent  $(Z_s)_{s \ge 0}$ in $\mathcal A^r_G$ defined as 
\begin{equation}\label{mirror_descent_flow}
\partial_s Z^{ij}_s=-\left(  V^{\boldsymbol{\pi}(Z),\boldsymbol{\lambda}_s}_j+G_{ij}- V^{\boldsymbol{\pi}(Z),\boldsymbol{\lambda}_s}_i+ \boldsymbol{\lambda}_s Z^{ij}_s \right),
\end{equation}
For its wellposedness, we have the following lemma.
\begin{lemma}\label{lem_policy_improment}
Given $Z_0 \in \mathcal A^r_G$, the mirror descent flow \eqref{mirror_descent_flow} is well-defined. Moreover, it holds that $\partial_s V^{\boldsymbol{\pi}(Z_s),\boldsymbol{\lambda}_s}_i(x) \le 0$  for all $i \in \mathbb I_m $ and $x \in \mathcal O$. 
\end{lemma}
The previous lemma states that the value function is decreasing along the flow. Thus, it is therefor not surprising that it should convergence to the optimal value function.. In fact, we can quantify the error bound by the following theorem.  
\begin{theorem}\label{thm_error_bound}
	For any $\lambda >0$ and non-increasing temperature schedule $\boldsymbol{\lambda} \in C^1([0,\infty),(0,\infty))$, let $\{Z_s\}_{s\ge0}$ be the corresponding mirror descent flow and $\bar{\boldsymbol{\pi}}$ be the optimal policy defined by \eqref{optim_cond} with temperature $\lambda$. Then, it holds that 
\begin{equation}\label{one_side_error_bound}
	\begin{split}
		&V^{\boldsymbol{\pi}(Z_s),\boldsymbol{\lambda}_s}_i(x)-V^{\bar{\boldsymbol{\pi}},\lambda}_i(x)\\
		\le& \frac{1}{\int_0^s e^{\int_0^u \boldsymbol{\lambda}_v dv}du} \mathbb E^{\bar{\boldsymbol{\pi}}}_{i,x}\left[\int_0^{\tau_{\mathcal O}}\sum_{j \neq  I_t}( D(\bar \pi^{ I_tj}( X_t)|\exp(Z^{ij}_0
		(\bar X_t))))e^{-rt}dt\right]\\
		&\qquad+\frac{\int_0^s  e^{\int_0^u \boldsymbol{\lambda}_v dv}(\boldsymbol{\lambda}_u-\lambda)du}{\int_0^s e^{\int_0^u \boldsymbol{\lambda}_v dv}du} \mathbb E^{\bar{\boldsymbol{\pi}}}_{i,x}\left[\int_0^{\tau_{\mathcal O}} (\sum_{j \neq I_t}D(\bar \pi^{I_tj}(X_t)|1))e^{-rt}dt\right].
	\end{split}
\end{equation}
\end{theorem}

From  Lemma \ref{lem_near_ineq} and \eqref{optim_cond}, we get that $\bar \pi^{ij}(x) \le \frac{C_{up}}{\lambda}$. Since $D(x|1)$ is convex with respect to $x$, it holds that 
$
\sup_{0<x<\frac{C_{up}}{\lambda}} D(x|1)=1 \vee D(\frac{C_{up}}{\lambda}|1) \le C(1+\frac1\lambda\log\frac1\lambda).  
$ Hence, 
$$
\mathbb E^{\bar{\boldsymbol{\pi}}}_{i,x}\left[\int_0^{\tau_{\mathcal O}} (\sum_{j \neq I_t}D(\bar \pi^{I_tj}|1))e^{-rt}dt\right] \le \frac{C}{r}(1+\frac1\lambda\log\frac1\lambda).
$$
One can have a similar estimation for the first expectation in \eqref{one_side_error_bound}. Thus, we get the following corollary.
\begin{corollary}\label{cor_error}
	Under the assumption of Theorem \ref{thm_error_bound}, it holds that 
	$$
	V^{\boldsymbol{\pi}(Z_s),\boldsymbol{\lambda}_s}_i(x)-V^{\bar{\boldsymbol{\pi}},\lambda}_i(x) \le C(1+\frac1\lambda \log \frac1\lambda)\left(\frac{1}{\int_0^s e^{\int_0^u \boldsymbol{\lambda}_v dv}du}+\frac{\int_0^s  e^{\int_0^u \boldsymbol{\lambda}_v dv}(\boldsymbol{\lambda}_u-\lambda)du}{\int_0^s e^{\int_0^u \boldsymbol{\lambda}_v. dv}du} \right).
	$$
\end{corollary}

Consider the constant schedule $\boldsymbol{\lambda}_s \equiv \lambda$. Recalling the error estimation \eqref{ineq_error} between regularized optimal value function and the original optimal value, one can get that 
\begin{equation*}
V^{\boldsymbol{\pi}(Z_s),{\lambda}}_i(x)-V^*_i(x) \le C\left(\frac{\lambda+\log\frac1\lambda}{e^{\lambda s}-1}+\lambda \log\frac1\lambda\right).	
\end{equation*}
Next, let us consider an annealing schedule $\boldsymbol{\lambda}_s=\frac{1}{\sqrt{1+s}}$. Then, we have the following convergence rate.
\begin{corollary}\label{cor_annealing} There exits a constant $C$ such that
$$
0 \le 	V^{\boldsymbol{\pi}(Z_s),\boldsymbol{\lambda}_s}_i(x)-V^{*}_i(x) \le C \frac{\log s}{\sqrt{s}}.
$$
\end{corollary}

The convergence results obtained above are closely related to those of Sethi, Siska, and Zhang \cite{sethi2025entropy}, who studied entropy annealing for policy mirror descent in continuous-time stochastic control. Their framework concerns absolutely continuous controls, and the policy variable is a probability density over the control space. In contrast, our setting involves singular switching controls, where the policy is parameterized by the transition intensities of a continuous-time Markov chain. Despite these structural differences, both problems admit a KL-regularized variational formulation and a mirror descent dynamics in logarithmic policy coordinates. Consequently, the convergence estimates share a common structure, consisting of an optimization error controlled by the mirror descent flow and a regularization bias arising from the entropy penalty. In particular, under a suitable annealing schedule, both approaches converge to the value function of the original unregularized control problem. These parallels indicate that entropy annealing and mirror descent constitute a robust optimization principle that extends beyond classical drift control to stochastic switching systems.

Finally, it is instructive to compare the mirror descent flow with the classical
policy iteration algorithm for optimal switching problems. Given a policy $\boldsymbol{\pi}$,
policy iteration alternates between a policy evaluation step, in which the linear
system \eqref{eq_value_function} is solved to obtain $V^{\boldsymbol{\pi},\lambda}$, and a policy improvement
step, in which the policy is replaced by the greedy optimizer
\[
\tilde \pi_{ij}(x)
=
\exp\!\left(
\frac{V^{\boldsymbol{\pi},\lambda}_i(x)-V^{\boldsymbol{\pi},\lambda}_j(x)-G_{ij}}{\lambda}
\right).
\]
Thus policy iteration may be viewed as performing a full minimization of the local
Bellman objective at each iteration. In contrast, the mirror descent flow
\eqref{mirror_descent_flow} performs only an infinitesimal policy improvement in the dual
logarithmic coordinates $Z_{ij}=\log \pi_{ij}$. The update direction is determined
by the Bellman residual
\[
V^{\boldsymbol{\pi},\lambda}_j+G_{ij}-V^{\boldsymbol{\pi},\lambda}_i+\lambda Z_{ij},
\]
and therefore can be interpreted as a damped or continuous-time analogue of
policy iteration. While policy iteration often exhibits fast local convergence, its
analysis typically relies on discrete improvement arguments and may be sensitive
to approximation errors. The mirror descent flow, by contrast, preserves the
KL-geometric structure induced by entropy regularization, admits a Lyapunov
functional given by the value function, and naturally accommodates annealing
schedules for the temperature parameter. From this perspective, policy iteration
and mirror descent may be viewed as two realizations of the same underlying
KL-regularized optimization principle: the former performs exact greedy policy
updates, whereas the latter follows a continuous gradient-like trajectory in the
space of Gibbs policies.

A remaining challenge concerns the robustness of both policy iteration and
mirror descent under stochastic approximation. In practical implementations,
the value functions and policy updates are computed from finite samples and are
therefore contaminated by random noise. The convergence results established
here are deterministic and correspond to the idealized infinite-sample regime.
It would be interesting to quantify how such perturbations accumulate along the
iterations or flow trajectory and to derive finite-sample error bounds for the
resulting algorithms. Understanding the interaction between sampling noise,
entropy regularization, and the switching structure is an important direction
for future work.
 \section{Proofs of all Main Results}
 \subsection{Proof of Theorem \ref{thm_hjb_system}}
 Given any $M$, consider a smooth, non-decreasing cut-off function $\phi_M$ such that $\phi_M(x)=\exp(x)$ for $x<M$ and $\phi_M(x)=M+1$, for $x>M+1$. Consider the following cut-off elliptic system
\begin{equation}\label{cut-off-hjb}
 \mathcal L_iV^{\lambda,M}(x)+L_i(x)-rV^{\lambda,M}_i=\lambda(\phi_M(\frac{V^{\lambda,M}_i-V^{\lambda,M}-G_{ij}}{\lambda})-1), x \in\mathcal O,
\end{equation}
 with boundary condition $V^{\lambda,M}_i(x)=0$ on $\partial \mathcal O$. Since the righthand side is bounded and Lipschitz continuous, classical results for elliptic system (see \cite{evans1984multiple}) yields that there exists a classical solution. Then, we shall prove that $V^{\lambda,M}_i$ satisfies the inequality \eqref{ineq_l_infty_esti}, which means that it is also the solution of \eqref{hjb_system} for sufficiently large $M$.

 First, note that the function $w:=V^{\lambda,M}_i-\frac{C_{\text{coef}}+(m-1)\lambda}{r}$ satisfies
 $$
 \mathcal L_iw-rw=-L_i(x)+\lambda\sum_{i \neq j}(\phi_M(\frac{V^{\lambda,M}_i-V^{\lambda,M}_j-G_{ij}}{\lambda})-1)+C_{\text{coef}}+(m-1)\lambda>0.
 $$  
 Applying maximum principle, one gets that $w\le 0$, which gives the second inequality in \eqref{ineq_l_infty_esti}. Next, noting that $(0,0,...,0)$ is a subsolution of the elliptic system \eqref{cut-off-hjb}, one can use comparison principle (see Theorem \ref{thm_comparison}) to get the first inequality of \eqref{ineq_l_infty_esti}.
 
 Having proved the well-posedness of \eqref{hjb_system}, an application of classical verification theorem yields that the policy $\bar {\boldsymbol{\pi}}$ defined by \eqref{optim_cond} is an optimal one. This finishes the proof.
 \subsection{Proof of Lemma \ref{lem_lip}}
 We first estimation $|D V^\lambda_i(x)|$ for $x \in \partial \mathcal O$. Noting that $V^\lambda_i$ satisfies the  boundary condition, one gets that $D_{\nu} V^\lambda_i(x)=0$ for any direction $\nu$ tangent to $\partial \mathcal O$ at $x$.  Hence, to get a bound on   $|D V^\lambda_i(x)|$, we just need an estimation for $D_n V^\lambda_i(x)$, where $n$ is the outer-norm vector of $\partial \mathcal O$ at $x$.  For that purpose, consider the solution $U_i$ of the following PDE
 \begin{equation}
 	\mathcal L_i U_i+L_i(x)-rU_i+\lambda=0,
 \end{equation}
 with boundary condition $U_i(x)=0$ for $x \in \partial \mathcal O$. Classical theory for elliptic PDEs yields that 
 $$
 |D U_i(x)| \le C.
 $$
 Moreover, applying comparison principle, one can get that $V_i \le U_i$. This implies that 
 $$
 D_n V^\lambda_i(x) \le D_n U_i(x) \le C.
 $$
 On the other hand, since $V^\lambda_i(x) \ge 0$ for any $x \in \mathcal O$, it holds that $D _n V^\lambda_i (x)\ge 0.$ This give a lower bound on  $D _n V^\lambda_i (x)$. 
 
Next, we estimate $|D V^\lambda_i(x)|$ for $x \in \mathcal O$ using Berstein method.  Denote by $W^{ij}=\exp(\frac{V^\lambda_i-V^\lambda_j-G_{ij}}{\lambda})$,
and define $w_i=\frac12|D V^\lambda_i(x)|^2$.
Noting that $w_i=\frac12\sum_{\gamma} (D_\gamma V^\lambda_i)^2 $, one can compute that 
$$
D_\alpha w_i=D_{\alpha \gamma}V^\lambda_i D_\gamma V^\lambda_i,\text{ and } D_{\alpha\beta}w_i=D_{\alpha\beta\gamma}V^\lambda_iD_\gamma V^\lambda_i +D_{\alpha\gamma} V^{\lambda}_iD_{\beta\gamma} V^\lambda_i.
$$
Hence,
\begin{equation*}
\mathcal L_i w_i=a^{\alpha\beta}D_{\alpha\beta\gamma}V^\lambda_i D_\gamma V^\lambda_i+a^{\alpha\beta}D_{\alpha\gamma}V^{\lambda}_iD_{\beta\gamma}V^\lambda_i+b^\alpha_i D_{\alpha\gamma}V^\lambda_i D_\gamma V^\lambda_i.
\end{equation*}
On the other hand, it holds that 
\begin{equation*}
D_\gamma \mathcal L_i V^\lambda_i= a^{\alpha\beta} D_{\alpha\beta \gamma}V^\lambda_i+D_\gamma a^{\alpha\beta} D_{\alpha\beta} V^\lambda_i +b^\alpha_i D_{\alpha\gamma}V^\lambda_i +D_\gamma b^\alpha_i D_\alpha V^\lambda_i. 
\end{equation*}
 Differentiating \eqref{hjb_system}, we see that 
 \begin{equation*}
 	\begin{split}
 	D_\gamma \mathcal L_i V^\lambda_i=r D_\gamma V^\lambda_i -D_\gamma L_i+W^{ij}(D_\gamma V^\lambda_i-D_\gamma V^\lambda_j).
 \end{split}
 \end{equation*}
Thus, 
\begin{equation*}
	\begin{split}
		\mathcal L_i w_i=a^{\alpha\beta}D_{\alpha\gamma}V^\lambda_i D_{\beta\gamma}V^\lambda_i-D_\gamma a^{\alpha\beta} D_{\alpha\beta}V^\lambda_i D_\gamma V^\lambda_i-D_\gamma b^\alpha D_\alpha V^\lambda_i D_\gamma V^\lambda_i\\
		+r(D_\gamma V^\lambda_i)^2-D_\gamma L_i D_\gamma V^\lambda_i+W^{ij}(D_\gamma V^\lambda_i-D_\gamma V^\lambda_j)D_\gamma V^\lambda_i
	\end{split}
\end{equation*} 
 From the assumptions on the coefficients, we have that 
 $$
 a^{\alpha\beta}D_{\alpha\gamma}V^\lambda_i D_{\beta\gamma}V^\lambda_i \ge \sigma_0 |D_{\alpha\gamma} V^\lambda_i|^2,
 $$
 \begin{equation*}
\begin{split}
	&-D_\gamma a^{\alpha\beta} D_{\alpha\beta}V^\lambda_i D_\gamma V^\lambda_i\ge -\|D_\gamma a\|_{\infty}|D_{\alpha\beta} V^\lambda_i||D_\gamma V^\lambda_i|\\
	\ge& -\frac{\sigma_0}{2} |D_{\alpha \beta} V^\lambda_i|^2 -\frac{\|D_\gamma a\|_\infty}{2\sigma_0}|D_\gamma V^\lambda_i|^2,
	\end{split}
 \end{equation*}
 and
 $$
 -D_\gamma b^\alpha_i D_\alpha V^\lambda_i D_\gamma V^\lambda_i \ge -\|D_\gamma b_i\|_{\infty} |D_\gamma V^\lambda_i|^2. 
 $$
 Moreover, by H\"older inequality,
 $$
 W^{ij}(D_\gamma V^\lambda_i-D_\gamma V^\lambda_j)D_\gamma V^\lambda_i \ge \frac12 W^{ij}(|D_\gamma V^\lambda_i|^2-|D_\gamma V^\lambda_j|^2).
 $$
Hence,  for some constants $C_1$, $(w_1,w_2,...,w_m)$ is a super-solution of the following system
\begin{equation*}
\begin{split}
	\mathcal L_i w_i \ge &(r-\frac{\|D_\gamma a\|_\infty}{2\sigma_0}-\|D_\gamma b_i\|_\infty)w_i-\|D_\gamma L_i\|_\infty w^{\frac12}_i+\frac12 W^{ij}(w_i-w_j)\\
	\ge &\frac{(r-\frac{\|D_\gamma a\|_\infty}{2\sigma_0}-\|D_\gamma b_i\|_\infty)}{2}w_i-C_1+W^{ij}(w_i-w_j).
\end{split}
\end{equation*}
However, at the current stage, we can not use the comparison principle directly, as the coefficient before the zero order term needs not to be positive. Thus, we need some further modification.

Consider the function $e^{\delta x_1}$ for some positive constant $\delta$ to be decided later. One can compute that 
$$
\mathcal L_i e^{\delta x_1}=(a^{11}\delta^2+b^1_i\delta)e^{\delta x_1} \ge (\sigma_0\delta^2-\|b_i\|_\infty \delta)e^{\delta x_1}. 
$$
Denote by $\tilde w_i=w_i e^{\delta x_1}$. Then, it holds that 
$$
\mathcal L_i \tilde w_i=(\mathcal L_i e^{\delta x_1})w_i+e^{\delta x_1}\mathcal L_i w_i+(a^{1\alpha}+a^{\alpha1})\delta e^{\delta x_1}D_\alpha w_i.
$$
Moreover, we have that 
$$
D_\alpha \tilde w_i=e^{\delta x_1}D_\alpha w_i +1_{\{\alpha=1\}}\delta e^{\delta x_1}w_i.
$$
Hence, 
\begin{equation*}
\begin{split}
&\mathcal L_i \tilde w_i-(a^{1\alpha}+a^{\alpha1})\delta D_\alpha \tilde w_i\\
=&(\mathcal L_i e^{\delta x_1})w_i+e^{\delta x_1}\mathcal L_i w_i-2a^{11}\delta e^{\delta x_1}w_i\\
\ge&(\sigma_0 \delta^2-\|b_i\|_\infty\delta-2a^{11}\delta+\frac{(r-\frac{\|D_\gamma a\|_\infty}{2\sigma_0}-\|D_\gamma b_i\|_\infty)}{2}) \tilde w_i-C_1e^{\delta x_1}+W^{ij}(\tilde w_i-\tilde w_j)
\end{split}
\end{equation*}
Choose $\delta$ large enough such that $\sigma_0 \delta^2-\|b\|_\infty\delta-2a^{11}\delta+\frac{(r-\frac{\|D_\gamma a\|_\infty}{2\sigma_0}-\|D_\gamma b_i\|_\infty)}{2}>1$. Then, $\tilde w_i$ is a sub-solution of 
$$
\mathcal L_i \tilde w_i- (a^{1\alpha}+a^{\alpha1})\delta D_\alpha \tilde w_i\ge \tilde w_i-C_1e^{\delta x_1}+W^{ij}(\tilde w_i-\tilde w_j).
$$ 
It also easy to see that, for some sufficiently large constant $C_2$, $(C_2,C_2,...,C_2)$ is a super-solution. Hence, using comparison principle\footnote{There is an additional first order term, but the comparison principle still holds.}, we get that $\tilde w_i \le C_2$, which gives the desired result.
  \subsection{Proof of Lemma \ref{lem_near_ineq}}
  For any $i, j \in I_m$, define $w:=V^\lambda_i-V^\lambda_j$.  
  Then, from \eqref{hjb_system}, it holds that 
  \begin{equation*}
  	\begin{split}
  		\mathcal L_i w=&rw+L_j(x)-L_i(x)-(b^\alpha_i-b^\alpha_j)D_\alpha V^\lambda_j\\
  		&+\lambda\left(\exp(\frac{V^\lambda_i-V^\lambda_j-G_{ij}}{\lambda})-\exp(\frac{V^\lambda_j-V^\lambda_i-G_{ji}}{\lambda})\right)\\
  		&+\lambda \sum_{k\neq i,j} \exp(\frac{V^\lambda_i-V^\lambda_k-G_{ik}}{\lambda})-\exp(\frac{V^\lambda_j-V^\lambda_k-G_{jk}}{\lambda})
  	\end{split}
  \end{equation*}
  Assume that there exists $x^* \in \mathcal O$ such that 
  $$
  w(x^*)=\sup_{x \in \mathcal O}w(x)>G_{ij}+\lambda \log\frac{C_{up}}{\lambda}. 
  $$
Then,  it holds that $\mathcal L_i w(x^*) \le 0$.   It is easy to see that, at this point $x^*$, 
$$
\exp(\frac{V^\lambda_j-V^\lambda_i-G_{ji}}{\lambda})=\exp(-\log{\frac {C_{up}}{\lambda}}) \le 1,
$$
and
\begin{equation*}
\begin{split}
	&\exp(\frac{V^\lambda_i-V^\lambda_k-G_{ik}}{\lambda})-\exp(\frac{V^\lambda_j-V^\lambda_k-G_{jk}}{\lambda})\\
	=&\exp(\frac{V^\lambda_i-V^\lambda_k-G_{ik}}{\lambda})(1-\exp(\frac{-w+G_{ik}-G_{jk}}{\lambda}))\\
	=&\exp(\frac{V^\lambda_i-V^\lambda_k-G_{ik}}{\lambda})(1-\exp(\frac{-(w-G_{ij})+G_{ik}-G_{jk}-G_{ij}}{\lambda})) >0,
\end{split}
\end{equation*}
where the last inequality due to fact that $w(x^*)>G_{ij}$ and $G_{ij}+G_{jk} >G_{ik}$.  Hence, 
$$
0 \ge \mathcal L_i w(x^*) \ge rw(x^*)-2C_{coef}(1+C_{Lip})+\lambda(\exp(\frac{w(x^*)-G_{ij}}{\lambda})),
$$
which is a contradiction. 
\subsection{Proof of Theorem \ref{thm_diff_V}}
First, since $V^*_i$ is the solution of \eqref{VI_sys}, it holds that 
$$
\mathcal L_i V^*_i +L_i(x)-rV^*_i \ge 0,
$$
and 
$$
V^*_i\le V^*_j+G_{ij}.
$$
Hence, it is a sub-solution of \eqref{hjb_system}, i.e.,
$$
\mathcal L_i V^*_i+L_i(x)-rV^*_i \ge \lambda\sum_{i\neq j}(\exp(\frac{V^*_i-V^*_j-G_{ij}}{\lambda})-1).
$$ 
From the comparison principle, it holds that $V^\lambda_i \ge V^*_i$. 

Next, we prove the first inequality by contradiction. Let $\iota:=\frac{(1-\kappa(\lambda))C_{coef}+(m-1)\lambda}{r}$ and  assume that 
$$
\inf_{i\in I_m,x \in \mathcal O} V^*_i-\kappa(\lambda)V^\lambda_i+\iota=l<0,
$$
and it attains at some point $(i,z)$. 

If $z \in C_{i}:=\left\{x \in \mathcal O|V^*_i(x) <V^*_j(x)+G_{ij},\text{ for any $j \neq i$}\right\}$,  then 
$$
\mathcal L_i V^*_{i}(z)=rV^*(z)-L_i(z),
$$
and
\begin{equation}
	\begin{split}
0 &\le \mathcal L_i V^*_i(z)-\kappa(\lambda)\mathcal L_iV^\lambda_i(z)\\
=&r(V^*(z)-\kappa(\lambda)V^\lambda_i(z))-(1-\kappa(\lambda))L_i(z)-\lambda\kappa(\lambda)\sum_{j\neq i}(\exp(\frac{V^\lambda_i-V^\lambda_j-G_{ij}}{\lambda})-1)\\
\le &r(l-\iota)+(1-\kappa(\lambda))C_{coef}+(m-1)\lambda<0,
\end{split}
\end{equation}
which is a contradiction. Hence, we must have $z  \notin C_i$, which means that there exists an index, say $j$, such that 
$$
V^*_i(z)=V^*_j(z)+G_{ij}.
$$
Then,  
\begin{equation*}
\begin{split}
	&V^*_j(z)+G_{ij}=\kappa(\lambda)V^\lambda_i(z)-\iota+l \\
	\le& \kappa(\lambda)(V^j_\lambda+G_{ij}+\lambda\log \frac {C_{up}}{\lambda})-\iota+l \\
	\le &\kappa(\lambda)V^\lambda_j(z)+G_{ij}-\iota+l,
\end{split}
\end{equation*}
where the first inequality is due to the result of Lemma \ref{lem_near_ineq} and the second is due to the definition of $\kappa(\lambda)$. Hence, $(j,z)$ is also a minimum point. Repeating previous argument, one can find a sequence of index $(i_1,i_2,...,i_p)$ such that $i_p=i_1$ and, for $k=1,2,...,p-1$,
$$
V^*_{i_k}(z)=V^*_{i_{k+1}}(z)+G_{i_{k}i_{k+1}}
$$
Summing these equations, one gets that 
$$
\sum_{k=1}^{p-1}G_{i_{k}i_{k+1}}=0,
$$
which is a contradiction according to our assumption on $G$. 
\subsection{Proof of Lemma \ref{lem_performance_diff}}
Since the value functions belongs to $W^{2,p^*}$ and $tr(\sigma\sigma^T D^2\varphi) \in L^\infty$, one can apply It\^o formula related to continuous time Markov chain and take expectation to get that, for $\varphi_i=V^{\boldsymbol{\pi},\lambda}_i$ and $V^{\tilde{\boldsymbol{\pi}},\lambda}_i$,   
\begin{equation}\label{eq_ito_expec}
\varphi_i(x)=-\mathbb E_{i,x}^{\boldsymbol{\pi}}\left[\int_0^{\tau_{\mathcal O}}   \left( \mathcal L_{I_t}\varphi_{I_t}(X_t)+\sum_{j\neq I_t}(\varphi_{j}-\varphi_{I_t})(X_t)\pi^{I_tj} \right)e^{-rt}dt\right]
\end{equation}
Note that  $V^{\boldsymbol{\pi},\lambda}_i$ and $V^{\tilde{\boldsymbol{\pi}},\lambda}_i$ are solutions of \eqref{eq_value_function} corresponding policies $\boldsymbol{\pi}$ and $\tilde{\boldsymbol{\pi}}$. It holds that 
\begin{equation*}
\begin{split}
&\mathcal L_i V^{\tilde{\boldsymbol{\pi}},\lambda}_i-\mathcal L_iV^{{\boldsymbol{\pi}},\lambda}_i\\
=&\sum_{j\neq i} (V^{{\boldsymbol{\pi}},\lambda}_j+G_{ij}-V^{\boldsymbol{\pi},\lambda}_i)\pi^{ij}+\lambda D(\pi^{ij}|1)-\sum_{j\neq i} (V^{\tilde{\boldsymbol{\pi}},\lambda}_j+G_{ij}-V^{\tilde{\boldsymbol{\pi}},\lambda}_i)\tilde \pi^{ij}-\lambda D(\tilde \pi^{ij}|1)\\
=&(V^{{\boldsymbol{\pi}},\lambda}_j-V^{\boldsymbol{\pi},\lambda}_i-V^{\tilde{\boldsymbol{\pi}},\lambda}_j+V^{\tilde{\boldsymbol{\pi}},\lambda}_i)\pi^{ij}+(V^{\tilde{\boldsymbol{\pi}},\lambda}_j+G_{ij}-V^{\tilde{\boldsymbol{\pi}},\lambda}_i+\lambda \log \tilde \pi^{\ij})( \pi^{ij}-\tilde \pi^{ij})\\
&\qquad+\lambda D(\pi^{ij}|\tilde \pi^{ij}),
\end{split}
\end{equation*}
where we use the equation that 
$$
D(x|1)-D(y|1)=(x-y)\log y+D(x|y).
$$
Using \eqref{eq_ito_expec} and above equation, we get the desired result.  
\subsection{Proof of Lemma \ref{lem_policy_improment}}
Denote by $Y:=\mathcal C(\bar{\mathcal O},\mathbb R^m)$ equipped with supreme norm. Let us define a mapping $\mathcal T^\lambda$ from $\mathcal A^r_G$ to $Y$. For any $Z \in \mathcal A^r_G$, $\mathcal T^\lambda(Z)=(V^{\boldsymbol{\pi}(Z),\lambda}_1,V^{\boldsymbol{\pi}(Z),\lambda}_2,...,V^{\boldsymbol{\pi}(Z),\lambda}_m)$ is the solution of the following linear system:
 \begin{equation}\label{eq_value_function_Z}
	\mathcal L_i V^{\boldsymbol{\pi}(Z),\lambda}_i+L_i(x)-r V^{\boldsymbol{\pi}(Z),\lambda}_i+\sum_{ j \neq i}(V^{\boldsymbol{\pi}(Z),\lambda}_j+G_{ij}-V^{\boldsymbol{\pi}(Z),\lambda}_i)\pi^{ij}+\lambda D(\exp(Z^{ij})|1)=0,  
\end{equation}
with the boundary condition $V^{\boldsymbol{\pi}(Z),\lambda}_i(x)=0$ on $\partial \mathcal O$.  For any $m>1$, define a subset $U_m \subset \mathcal A^r_G$ such that 
$$
U_m=\left\{ Z \in \mathcal A^r_G | Z^{ij} (x)\le m\right\}. 
$$
We decompose the proof in the following steps.

{\it Step 1.} For any $Z \in U_m$ and $\lambda \le \lambda_0$, $0 \le V^{\boldsymbol{\pi}(Z),\lambda}_i(x) \le C_1(1+me^m)$, where the constants $C_1$ depends on the coefficients and $\lambda_0$. In fact, since $Z \in U_m$, it holds that $\pi^{ij}=\exp(Z^{ij})$ is bounded by $e^m$. Thus, $V^{\boldsymbol{\pi},\lambda}_i$ is the solution of a linear elliptic system with bounded coefficients. One see that $0$ and $C_1(1+me^m)$ are sub and super solution respectively for some constant $C_1$ sufficiently large. Hence, we have the claim using comparison principle.

{\it Step 2.}  We prove that $\mathcal T^\lambda$ is locally Lipschitz in $Z$, i.e.  for any $Z_1,Z_2 \in U_m $ and $\lambda \le \lambda_0$, 
\begin{equation}\label{ineq_local_Lip}
\sup_{i,x}|V^{\boldsymbol{\pi}(Z_1),\lambda}_i(x)-V^{\boldsymbol{\pi}(Z_2),\lambda}_i(x)| \le C_m\sup_{i\neq j,x}|Z^{ij}_1(x)-Z^{ij}_2(x)|,
\end{equation}
for some constant $C_m$ depending on  the coefficients, $\lambda_0$ and $m$.

Denote by $\Delta_i:=V^{\boldsymbol{\pi}(Z_1),\lambda}_i-V^{\boldsymbol{\pi}(Z_2),\lambda}_i$. We see that it solves
\begin{equation*}
\begin{split}
&\mathcal L_i \Delta_i-r\Delta_i-\sum_{j \neq i}(\Delta_i-\Delta_j)\exp(Z^{ij}_1)\\
=&\sum_{j \neq i}(V^{\boldsymbol{\pi}(Z_2),\lambda}_j+G_{ij}-V^{\boldsymbol{\pi}(Z_2),\lambda}_i)(\exp(Z^{ij}_2)-\exp(Z^{ij}_1))\\
&\qquad+\lambda(D(\exp(Z^{ij}_2)|1)-(D(\exp(Z^{ij}_1)|1))
\end{split}
\end{equation*}
with boundary condition $\Delta_i=0$ on $\partial \mathcal O$. Using the claim of Step 1, it holds that the right hand side is bounded by $C_m\sup_{i\neq j,x}|Z^{ij}_1(x)-Z^{ij}_2(x)|$  for  some constant $C_m$ depending on  the coefficients, $\lambda_0$ and $m$. Thus, one can get \eqref{ineq_local_Lip} using comparison principle once again.

{\it Step 3.} Consider a function $\chi_n$ as $\chi_m(x) =-x \wedge m$. 
Then, one can define a functional $\mathcal U^\lambda$  from $U_m$ to itself as 
$$
(\mathcal U^\lambda(Z))^{ij}=\chi_m(V^{\boldsymbol{\pi}(Z),\lambda}_j+G_{ij}-V^{\boldsymbol{\pi}(Z),\lambda}_i+\lambda Z^{ij}).
$$
Assume that there exists $Z \in C^1([0,S],U_m)$ for some $S>0$, such that 
$$
\partial_s Z_s=\mathcal U^{\boldsymbol{\lambda}_s}(Z_s).
$$
Then, we claim that $\partial_s V^{\boldsymbol{\pi}(Z_s),\boldsymbol{\lambda}_s}_i(x)\le 0$ for any $i \in \mathbb I_m$ and $x \in \mathcal O$.

To prove this, we use Lemma \ref{lem_performance_diff} to get that
$$
V^{\boldsymbol{\pi}(Z_s),\boldsymbol{\lambda}_s}_i(x)-V^{{\boldsymbol{\pi}(Z_{s+h})},\boldsymbol{\lambda}_{s+h}}_i(x)=\boldsymbol{\lambda}_s I_1+I_2+I_3,
$$ 
where 
$$
I_1= \mathbb E_{i,x}^{\boldsymbol{\pi}(Z_s)}\left[ \int_0^{\tau_{\mathcal O}}\sum_{j \neq I_t}(D(\exp(Z^{I_tj}_s)(X_t))|\exp(Z^{I_tj}_{s+h})(X_t)) )e^{-rt}dt\right],
$$
\begin{equation*}
	\begin{split}
I_2=\mathbb E^{ {\boldsymbol{\pi}(Z_s)}}_{i,x}\bigg[\int_0^{\tau_{\mathcal O}}  \sum_{j \neq  I_{t}}&\left(V^{{\boldsymbol{\pi}(Z_{s+h})},\boldsymbol{\lambda}_s}_j+G_{I_tj}-V^{{\boldsymbol{\pi}(Z_{s+h})},\boldsymbol{\lambda}_s}_i+\boldsymbol{\lambda}_s \log Z^{I_tj}_{s+h}\right)(X_t)\\
&\cdot\left(\exp(Z^{I_tj}_s(X_t))-(\exp(Z^{I_tj}_{s+h}(X_t))\right)e^{-rt}dt\bigg],
\end{split}
\end{equation*}
and
\begin{equation*}
	\begin{split}
		I_3&=V^{\boldsymbol{\pi}(Z_{s+h}),\boldsymbol{\lambda}_s}_i(x)-V^{{\boldsymbol{\pi}(Z_{s+h})},\boldsymbol{\lambda}_{s+h}}_i(x)\\
		&=(\boldsymbol{\lambda}_{s}-\boldsymbol{\lambda}_{s+h})\mathbb E^{ {\boldsymbol{\pi}(Z_{s+h})}}_{i,x}\left[ \int_0^{\tau_{\mathcal O}}\sum_{j \neq I_t}D(\exp(Z^{I_tj}_{s+h})(X_t)|1)e^{-rt}dt \right].
		\end{split}
	\end{equation*}
Noting that $\partial_y D(x|y)=1-\frac xy$, direct computation implies that  $\lim_{h \rightarrow 0} \frac{I_1}{h}=0$. It is also obvious that 
$$
\lim_{h \rightarrow 0} \frac{I_3}{h}=-\dot{\boldsymbol{\lambda}}_s\mathbb E^{ {\boldsymbol{\pi}(Z_{s+h})}}_{i,x}\left[ \int_0^{\tau_{\mathcal O}}\sum_{j \neq I_t}D(\exp(Z^{I_tj}_{s+h})(X_t)|1)e^{-rt}dt \right] \ge 0.
$$ 
For the middle term, we see that the limit of $\lim_{h \rightarrow 0} \frac{I_2}{h}$ goes to 
\begin{equation*}
\begin{split}
\mathbb E^{ {\boldsymbol{\pi}(Z_s)}}_{i,x}\left[\int_0^{\tau_{\mathcal O}}\left(V^{{\boldsymbol{\pi}(Z_{s})},\boldsymbol{\lambda}_s}_j+G_{I_tj}-V^{{\boldsymbol{\pi}(Z_{s})},,\boldsymbol{\lambda}_s}_i+\boldsymbol{\lambda}_s \log Z^{I_tj}_{s}\right)\partial_s Z^{I_tj}_s \exp(Z^{I_tj}_s) e^{-rt}dt\right].
	\end{split}
\end{equation*}
Then, from the definition of $\mathcal U^\lambda$, we see that the integrand is always non-negative, which implies that $\lim_{h \rightarrow 0} \frac{I_2}{h} \ge 0$. Combining all these together, we prove the claim.

{\it Step 4.}  At last, let us prove the well-posedness of the mirror descent flow for any fixed time interval $[0,S]$.  For initial state $Z_0 \in \mathcal A^r_G$, denote by $M:=\sup_{i\neq j,x}  |Z^{ij}_0(x)|+\sup_{i,x}V^{\boldsymbol{\pi}(Z_0),\boldsymbol{\lambda}_0}_i(x)+\sup_{i \neq j}G_{ij}$ and choose $m=(3+\lambda_0+\frac{1}{\lambda_S})M+1 $. Recall that we have shown that $\mathcal U^{\lambda}$ is Lipschitz continuous on $U_m$ with Lipschitz constant $C_m$ for any $\lambda \le \lambda_0$. Denote by $X_{\varepsilon}:=C^1([0,\varepsilon],U_m)$ with $\varepsilon=\frac{1}{m+C_m}$. Then, one can define a mapping $\Phi$  from $X_\varepsilon$ to itself as 
$$
\Phi(X_\cdot)_s=Z_0+\int_0^s \mathcal U^{\boldsymbol{\lambda}_s}(X_u)du,
$$
which is a contraction. Thus, $\Phi$ admits a unique fixed point $\tilde Z$, which is a solution of
$$
\partial_s \tilde Z_s=\mathcal U^{\boldsymbol{\lambda}_s}(\tilde Z_s). 
$$
Moreover, it holds that $\partial_s V^{\boldsymbol{\pi}(\tilde Z_s),\boldsymbol{\lambda}_s}_i(x) \le 0$.

Next, we want to show that  the solution of above equation can be extended beyond the interval $[0,\varepsilon]$. For that purpose, we need some a priori estimate for $\tilde Z$. Since $0 \le V^{\boldsymbol{\pi}(\tilde Z_s),\boldsymbol{\lambda}_s}_i \le V^{\boldsymbol{\pi}(\tilde Z_0),\lambda_0}_i$, it holds that 
$$
\partial _s \tilde Z_s\le (M-\boldsymbol{\lambda}_s \tilde Z_s) \wedge m  \le (M-\boldsymbol{\lambda}_s \tilde Z_s).
$$
Applying Gronwall's lemma, it holds that 
\begin{equation}\label{Z_up_bound}
	\tilde Z_s \le M+\int_0^s e^{-\int_u^s \boldsymbol{\lambda}_v dv} Mdu \le \frac{M}{\boldsymbol{\lambda}_s}+M \le m.
\end{equation}
On the other hand, 
$$
\partial_s \tilde Z_s \ge (-M-\boldsymbol{\lambda}_s \tilde Z_s)\wedge m.
$$
Consider  a function $z_s:=-M-\int_0^s e^{-\int_u^s \boldsymbol{\lambda}_v dv}Mdu $.  As $\boldsymbol{\lambda}$ is non-increasing, it holds that $z_s \ge -M-\frac{M}{\boldsymbol{\lambda}_s}$. Thus, $-M-\boldsymbol{\lambda_s}z_s \le M<m$,  which implies that $z_s$ is the solution of 
$$
\partial_s z_s=(-M-\boldsymbol{\lambda}_s z_s) \wedge m.
$$
From comparison principle of ODE, it holds that 
\begin{equation}\label{Z_lower_bound}
	\tilde Z_s \ge z_s \ge -m.
\end{equation}
Thus, $\tilde Z_s$ is uniformly bounded. Starting from $\varepsilon$, one can repeat previous argument to extend the solution $\tilde Z$ to the interval $[\varepsilon,2\varepsilon]$ and \eqref{Z_up_bound} and \eqref{Z_lower_bound} also hold. Repeating this procedure, we obtain that there exists  $\tilde Z \in  C^1([0,S],U_m)$ satisfying
$$
\partial_s \tilde Z_s=\mathcal U^{\boldsymbol{\lambda}_s}(Z_s).
$$  
However, from previous a priori estimate, it holds that 
\begin{equation*}
	\begin{split}
		&V^{\boldsymbol{\pi}(\tilde Z_s),\boldsymbol{\lambda}_s}_j(x)+G_{ij}-V^{\boldsymbol{\pi}(\tilde Z_s),\boldsymbol{\lambda}_s}_i(x)+\boldsymbol{\lambda}_s \tilde Z^{ij}_s \\
		\le& V^{\boldsymbol{\pi}(Z_0),\boldsymbol{\lambda}_0}_j(x)+G_{ij}+(1+\boldsymbol{\lambda}_s)M \le m.\\
		\end{split}
\end{equation*}
Hence, $\tilde Z$ is the mirror descent flow, which finishes the proof.
\subsection{Proof of Theorem \ref{thm_error_bound}}
Denote by 
$$
\mathcal D_{i,x}(\bar{\boldsymbol{\pi}}|\boldsymbol{\pi}(Z_s)):=\mathbb E^{\bar{\boldsymbol{\pi}}}_{i,x}\left[\int_0^{\tau_{\mathcal O}}\sum_{j \neq  I_t}( D(\bar \pi^{ I_tj}( X_t)|\exp(Z^{ij}_s
( X_t))))e^{-rt}dt\right].
$$
Taking derivative with respect to $s$, one can get that 
\begin{equation*}
\partial_s 	\mathcal D_{i,x}(\bar{\boldsymbol{\pi}}|\boldsymbol{\pi}(Z_s))=\mathbb E_{i,x}^{\bar{\boldsymbol{\pi}}}\left[ \int_0^{\tau_\mathcal O} \sum_{j \neq  I_t}\left( \exp(Z^{ I_tj}_s)(X_t)-\bar \pi^{ I_t j}(X_t)\right)\partial_s Z^{ I_tj}_s( X_t)  e^{-rt}dt\right].
\end{equation*}
From the definition of mirror descent flow \eqref{mirror_descent_flow} and Lemma \ref{lem_performance_diff}, the right hand side of above equation equals $V^{\bar{ \boldsymbol{\pi}},\boldsymbol{\lambda}_s}_i(x)-V^{\boldsymbol{\pi}(Z_s),\boldsymbol{\lambda}_s}_i(x)-\boldsymbol{\lambda}_s \mathcal D_{i,x}(\bar {\boldsymbol{\pi}}|\boldsymbol{\pi})(Z_s)$. Thus, 
$$
\partial_s 	\mathcal D_{i,x}(\bar{\boldsymbol{\pi}}|\boldsymbol{\pi}(Z_s))=V^{\bar{ \boldsymbol{\pi}},\boldsymbol{\lambda}_s}_i(x)-V^{\boldsymbol{\pi}(Z_s),\boldsymbol{\lambda}_s}_i(x)-\boldsymbol{\lambda}_s \mathcal D_{i,x}(\bar {\boldsymbol{\pi}}|\boldsymbol{\pi})(Z_s).
$$
Denote $\Lambda_t=e^{\int_0^t \boldsymbol{\lambda}_sds}$. One can solve above ODE as  
$$
\Lambda_s \mathcal D_{i,x}(\bar{\boldsymbol{\pi}}|\boldsymbol{\pi}(Z_s))-\mathcal D_{i,x}(\bar{\boldsymbol{\pi}}|\boldsymbol{\pi}(Z_0))=\int_0^s  \Lambda_u (V^{\bar{\boldsymbol{\pi}},\boldsymbol{\lambda}_u}_i(x)-V^{\boldsymbol{\pi}(Z_u),\boldsymbol{\lambda}_u}_i(x))  du.
$$
Then, we have 
\begin{equation*}
\begin{split}
&V^{\bar{\boldsymbol{\pi}},\boldsymbol{\lambda}_u}_i(x)-V^{\boldsymbol{\pi}(Z_u),\boldsymbol{\lambda}_u}_i(x)\\
=&V^{\bar{\boldsymbol{\pi}},\boldsymbol{\lambda}_u}_i(x)-V^{\bar{\boldsymbol{\pi}},\lambda}_i(x)+V^{\bar{\boldsymbol{\pi}},\lambda}_i(x)-V^{\boldsymbol{\pi}(Z_u),\boldsymbol{\lambda}_u}_i(x) \\
\le& V^{\bar{\boldsymbol{\pi}},\boldsymbol{\lambda}_u}_i(x)-V^{\bar{\boldsymbol{\pi}},\lambda}_i(x)+V^{\bar{\boldsymbol{\pi}},\lambda}_i(x)-V^{\boldsymbol{\pi}(Z_s),\boldsymbol{\lambda}_s}_i(x)\\
=&(\boldsymbol{\lambda}_u-\lambda) \mathbb E^{\bar{\boldsymbol{\pi}}}_{i,x}\left[\int_0^{\tau_{\mathcal O}} (\sum_{j \neq I_t}D(\bar \pi^{I_tj}|1))e^{-rt}dt\right]+ V^{\bar{\boldsymbol{\pi}},\lambda}_i(x)-V^{\boldsymbol{\pi}(Z_s),\boldsymbol{\lambda}_s}_i(x)	
\end{split}
\end{equation*}
where the inequality due to fact that $V^{\boldsymbol{\pi}(Z_u),\boldsymbol{\lambda}_u}_i(x)$ is non-increasing with respect to $u$. Noting that $\mathcal D_{i,x}(\bar{\boldsymbol{\pi}}|\boldsymbol{\pi}(Z_s)) \ge 0$, one gets 
\begin{equation*}
	\begin{split}
&V^{\boldsymbol{\pi}(Z_s),\boldsymbol{\lambda}_s}_i(x)-V^{\bar{\boldsymbol{\pi}},\lambda}_i(x)\\
\le& \frac{1}{\int_0^s e^{\int_0^u \boldsymbol{\lambda}_v dv}du} \mathbb E^{\bar{\boldsymbol{\pi}}}_{i,x}\left[\int_0^{\tau_{\mathcal O}}\sum_{j \neq  I_t}( D(\bar \pi^{ I_tj}( X_t)|\exp(Z^{ij}_s
(X_t))))e^{-rt}dt\right]\\
&\qquad+\frac{\int_0^s  e^{\int_0^u \boldsymbol{\lambda}_v dv}(\boldsymbol{\lambda}_u-\lambda)du}{\int_0^s e^{\int_0^u \boldsymbol{\lambda}_v dv}du} \mathbb E^{\bar{\boldsymbol{\pi}}}_{i,x}\left[\int_0^{\tau_{\mathcal O}} (\sum_{j \neq I_t}D(\bar \pi^{I_tj}(X_t)|1))e^{-rt}dt\right].
\end{split}
\end{equation*}
\subsection{Proof of Corollary \ref{cor_annealing}}
From Theorem \ref{thm_error_bound} and the definition of optimal value function, one immediately gets that 
$$
V^{\boldsymbol{\pi}(Z_s),\boldsymbol{\lambda}_s}_i(x) \ge V^{\boldsymbol{\lambda}_s}_i(x)\ge V^*_i(x).
$$
This gives the first inequality. Then, 
$$
V^{\boldsymbol{\pi}(Z_s),\boldsymbol{\lambda}_s}_i(x)- V^*_i(x)=V^{\boldsymbol{\pi}(Z_s),\boldsymbol{\lambda}_s}_i(x)-V^{\boldsymbol{\lambda}_s}_i(x)+V^{\boldsymbol{\lambda}_s}_i(x)-V^*_i(x).
$$
For the second term, we use \eqref{ineq_error} to get that 
$$
V^{\boldsymbol{\lambda}_s}_i(x)-V^*_i(x) \le C \boldsymbol{\lambda}_s \log \frac{1}{\boldsymbol{\lambda}_s} \le C \frac{\log s}{\sqrt{s}}.
$$  
For the first term, we shall utilize the result of Corollary \ref{cor_error}. One can calculate that (see also the proof of Theorem 2.17 in \cite{sethi2025entropy})
$$
\int_0^s e^{\int_0^u \boldsymbol{\lambda_v}dv}du=\frac12\left( e^{2\sqrt{s+1}-2}(2\sqrt{s+1}-1)-1\right),
$$   
and 
$$
\int_0^s \boldsymbol{\lambda}_ue^{\int_0^u \boldsymbol{\lambda_v}dv}du=e^{2\sqrt{s+1}-2}-1.
$$
Hence,
$$
\frac{1}{\int_0^s e^{\int_0^u \boldsymbol{\lambda_v}dv}du} \le Ce^{-2\sqrt{s+1}},
$$
and
$$
\frac{\int_0^s (\boldsymbol{\lambda}_u-\boldsymbol{\lambda}_s)e^{\int_0^u \boldsymbol{\lambda_v}dv}du}{\int_0^s e^{\int_0^u \boldsymbol{\lambda_v}dv}du}=\frac{e^{2\sqrt{s+1}-2}-1}{\frac12\left( e^{2\sqrt{s+1}-2}(2\sqrt{s+1}-1)-1\right)}-\frac{1}{\sqrt{s+1}} \le  \frac{C}{s}.
$$
This implies that 
$$
V^{\boldsymbol{\pi}(Z_s),\boldsymbol{\lambda}_s}_i(x) \le C\frac{\log s}{\sqrt{s}}.
$$
Combining two estimations together, we get the desired result.
 \appendix
 \section{A Comparison Principle for Elliptic Systems}
 The following comparison principle is a direct consequence of
 \cite[Theorem 4.7]{ishii1991viscosity}. We state it here for convenience,
 since it is repeatedly used throughout the paper. 
 \begin{theorem}[Comparison Principle]\label{thm_comparison}
Let $(f_1,f_2,...,f_m):\mathbb R^m \rightarrow \mathbb R^m$ satisfy
\begin{enumerate}
	\item  $f_i(u_1,u_2,...,u_m)$ is non-decreasing with respect to $u_i$ and non-increasing with respect to $u_j$ for $j\neq i$;
	\item $f_i(u_1+v,u_2+v,...,u_m+v)=f_i(u_1,u_2,...,u_m)$ for any $v \in \mathbb R$. 
\end{enumerate}
Suppose that  $(U_1,U_2,...,U_m)$ (resp. $(V_1,V_2,...,V_m)$) is a super-solution (resp. sub-solution), i.e.
$$
\mathcal L_i U_i-rU_i \le f_i(U_1,U_2,...,U_m),
$$ 	
$$
\text{(resp. $\mathcal L_i V_i-rV_i \ge f_i(V_1,V_2,...,V_m)$).}
$$
 If $V_i \le U_i $ on $\partial \mathcal O$, then, $V_i \le U_i $ in $\bar{ \mathcal O}$.
 \end{theorem}
 \begin{proof}
It suffices to just verify the monotone condition (A3) in \cite{ishii1991viscosity}. More precisely, if $r,s \in \mathbb R^m$, there exists $j\in I_m$ such that 
 $$
 r_j-s_j=\max_i (r_i-s_i)>0,
 $$
 then $f_j(r_1,r_2,...,r_m) \ge f_j(s_1,s_2,...,s_m)$. To see  this, denote by $v=r_j-s_j$. Then, we have that 
 $$
 f_j(r_1,r_2,...,r_m)=f_j(r_1-v,r_2-v,...,r_m-v) \ge f_j(s_1,s_2,...,s_m),
 $$
 where the inequality due to the non-increasing assumption of $f_j$.
 \end{proof}
 In this paper, the nonlinear terms $f_i$ are of the following two forms
 $$
 f_i(u_1,u_2,...,u_m)=\lambda \sum_{i \neq j} \left(\exp(\frac{u_i-u_j-G_{ij}}{\lambda})-1\right),
 $$  
 and
 $$
 f_i(u_1,u_2,...,u_m)=\sum_{i\neq j}W^{ij}(u_i-u_j),
 $$
 with $W^{ij}\ge0$. Both satisfy the above assumptions, and therefore the comparison principle can be applied throughout the paper without further verification.
\bibliographystyle{plain}
\bibliography{references}
\end{document}